\documentclass[12pt]{amsart}
\usepackage{amssymb}
\usepackage{amsthm}

\usepackage{amscd}
\usepackage{extarrows}
\usepackage{graphicx}
\usepackage{fancyhdr}
\usepackage{epsfig}
\usepackage{amsfonts}
\usepackage{mathrsfs}
\usepackage{picinpar} \usepackage{amsmath} \usepackage[all]{xy}
\usepackage[colorlinks=true]{hyperref}
\hypersetup{urlcolor=blue, citecolor=red}
\theoremstyle{plain}
\newtheorem{theorem}{Theorem}[section]

\newtheorem{lemma}[theorem]{Lemma}
\newtheorem{proposition}[theorem]{Proposition}
\theoremstyle{definition}
\newtheorem{definition}[theorem]{Definition}

\newtheorem{remark}[theorem]{Remark}

\newcommand{\Ext}{\mbox{\rm Ext}}
\newcommand{\Hom}{\mbox{\rm Hom}}

\newcommand{\Ker}{\mbox{\rm Ker}}
\newcommand{\SGF}{\mbox{\rm SGF}}
\newcommand{\SG}{\mbox{\rm SG}}
\newcommand{\GP}{\mbox{\rm GP}}
\newcommand{\GF}{\mbox{\rm GF}}
\newcommand{\G}{\mbox{\rm G}}

\newcommand{\SSG}{\mbox{\rm SSG}}

\newcommand{\GFI}{\mbox{\rm GFI}}
\newcommand{\p}{\mbox{\rm P}}
\newcommand{\F}{\mbox{\rm F}}
\newcommand{\I}{\mbox{\rm I}}

\begin{document}
\title{Stability of strongly Gorenstein flat modules}
\author{ZHANPING WANG \ \ \ \ ZHONGKUI LIU}

\footnote[0]{*Supported by National Natural Science Foundation of China (Grant No. 11201377, 11261050) and Program of Science and Technique of Gansu Province (Grant No. 1208RJZA145).}\footnote[0]{Address
correspondence to Zhanping Wang, Department of Mathematics, Northwest Normal University, Lanzhou 730070, PR China.}\footnote[0]{E-mail: wangzp@nwnu.edu.cn (Z.P. Wang),
liuzk@nwnu.edu.cn (Z.K. Liu).}

\date{}\maketitle
\hspace{6.3cm}\noindent{\footnotesize {\bf Abstract}
\vspace{0.2cm}

\hspace{-0.75cm} A left $R$-module $M$ is called two-degree strongly Gorenstein flat if there exists
an exact sequence $\cdots \longrightarrow D_{1}\longrightarrow D_{0}\longrightarrow D_{-1}\longrightarrow D_{-2}\longrightarrow \cdots$ of strongly Gorenstein flat left $R$-modules such that $M\cong\ker (D_{0}\longrightarrow D_{-1})$  and $\Hom_{R} (-, F)$ leaves the sequence exact for any flat (or Gorenstein flat) left $R$-module $F$. In this paper, we show that the two-degree strongly Gorenstein flat modules are nothing more than the strongly Gorenstein flat modules.

\vspace{0.2cm}
\noindent{\footnotesize {2010 {\it{Mathematics Subject
Classification}}:}
16D40, 16D50, 16E05, 16E30

\noindent{\footnotesize {{\it{Keywords and phrases}}:}  strongly Gorenstein flat modules, Gorenstein FP-injective modules, two-degree strongly Gorenstein flat modules, two-degree Gorenstein FP-injective modules, Gorenstein projective modules.

\section{Introduction}
Throughout this paper, $R$ denotes an associative ring with unity, and all modules are assumed to be a left $R$-module. Denote by P($R$), I($R$) and F($R$) the class of
all projective, injective and flat left $R$-modules respectively.

The development of the Gorenstein homological algebra has reached an advanced
level since the pioneering works of Auslander and Bridger(\cite{AB1969}). One of the
key points of this theory is its ability to identify Gorenstein rings. In the Gorenstein
homological algebra one replaces projective, injective and flat modules, the elementary
entities on which the classical homological algebra is based, with the Gorenstein
projective, Gorenstein injective and Gorenstein flat modules. Recall from \cite{EJ1995} that a left $R$-module $M$ is called Gorenstein projective if there is an exact sequence
\[\cdots \longrightarrow P_{1}\longrightarrow P_{0}\longrightarrow P_{-1}\longrightarrow P_{-2}\longrightarrow \cdots\]
of projective left $R$-modules such that $M\cong\ker (P_{0}\longrightarrow P_{-1})$  and $\Hom_{R} (-, \p(R))$ leaves the sequence exact. Dually, The Gorenstein injective modules are defined. In \cite{EJT1993}, the Gorenstein flat modules are defined in terms of the tensor product.

Recently, Sather-Wagstaff et al. \cite{SSW2008} introduced modules that we call two-degree
Gorenstein projective modules: a module $M$ is two-degree Gorenstein projective if
there exists an exact sequence
\[\cdots \longrightarrow G_{1}\longrightarrow G_{0}\longrightarrow G_{-1}\longrightarrow G_{-2}\longrightarrow \cdots\] of the Gorenstein projective modules
with $M\cong\ker (G_{0}\longrightarrow G_{-1})$ such that the functor $\Hom_{R}(-, G)$ and $\Hom_{R}(G, -)$ leave the sequence exact for any Gorenstein projective module $G$. They proved that any two-degree
Gorenstein projective module is nothing but a Gorenstein projective module (Theorem A in \cite{SSW2008}). Later, similar notions were introduced and studied in \cite{BK2012, SSW2011, YL2012}.

As a special case of Gorenstein projective module, Ding, Li and Mao introduced and studied in \cite{Ding2009} strongly Gorenstein flat module, and several well-known classes of rings are characterized in terms of these modules. A left $R$-module $M$ is called strongly Gorenstein flat if there is an exact sequence
\[\cdots \longrightarrow P_{1}\longrightarrow P_{0}\longrightarrow P_{-1}\longrightarrow P_{-2}\longrightarrow \cdots\]
of projective left $R$-modules with $M\cong\Ker (P_{0}\longrightarrow P_{-1})$ such that $\Hom_{R} (-, \F(R))$ leaves the sequence exact. Dually, Mao and Ding introduced and studied in \cite{Mao2008} Gorenstein FP-injective modules, and showed that there is a very close relationship between Gorenstein FP-injective modules and Gorenstein flat modules.
Since over a Ding-Chen ring (that is, a left and right coherent ring with finite left and right self FP-injective dimension) the strongly Gorenstein flat modules and Gorenstein FP-injective modules have many nice properties analogous to Gorenstein projective modules and Gorenstein injective modules over a Gorenstein ring (that is, a left and right noetherian ring with finite left and right self injective dimension), Gillespie \cite{Gillespie2010} renamed these modules as Ding projective modules and Ding injective modules, respectively. At the same time, Gillespie introduced the Ding flat modules but it turns out that they are nothing more than the Gorenstein flat modules by \cite[Lemma 2.8]{Mao2008}.

The main purpose of this paper is to establish the stability of the strongly Gorenstein flat modules under the very process used to define these entities.

When we submit this article we do not know Xu's work \cite{Xu2013}. After the article was already submitted, it was pointed out to us that our results were obtained by Xu. But Xu's proof of the main theorem \cite[Theorem A]{Xu2013} is completely different from ours.

\section{Main results}
According to \cite{Ding2009}, a left $R$-module $M$ is called strongly Gorenstein flat if there is an exact sequence
\[\cdots \longrightarrow P_{1}\longrightarrow P_{0}\longrightarrow P_{-1}\longrightarrow P_{-2}\longrightarrow \cdots\]
of projective left $R$-modules such that $M\cong\Ker (P_{0}\longrightarrow P_{-1})$ and $\Hom_{R}(-, \F(R))$ leaves the sequence exact. We use $\SGF(R)$ to denote the class of all strongly Gorenstein flat modules.

Note that every projective module is strongly Gorenstein flat, and every strongly Gorenstein flat module is Gorenstein projective. For a left coherent ring $R$, it follows from \cite[Proposition 10.2.6]{Enochs2000} that a finitely presented module is strongly Gorenstein flat if and only if it is Gorenstein projective. Clearly, every Gorenstein projective module over a left perfect ring is strongly Gorenstein flat. Also it follows easily from \cite[Corollary 4.6]{Gillespie2010} that every Gorenstein projective module over a Gorenstein ring is strongly Gorenstein flat.

Recall that a class of modules is called projectively resolving (injective coresolving) if it is closed under extensions and kernels of surjections (cokernels of injections), and it contains all projective (injective) modules.

 The strongly Gorenstein flat modules have the following properties.

\begin{lemma}\label{lem1} The following assertions hold.

$\mathrm{(1)}$ If $M\in \SGF(R)$, then $\Ext^{i}_{R}(M,L)=0$ for all $i>0$ and all module $L$ of finite flat dimension.

$\mathrm{(2)}$ $\SGF(R)$ is a projectively resolving class, and closed under direct sums and direct summands.
\end{lemma}

\begin{proof} (1) It is trivial.

(2) It follows by analogy with the proof of Theorem 2.5 in \cite{Holm2004}.
\end{proof}

\begin{definition}\label{def1} A module $M$ is called two-degree strongly Gorenstein flat if there exists an  exact sequence \[\cdots \longrightarrow D_{1}\longrightarrow D_{0}\longrightarrow D_{-1}\longrightarrow D_{-2}\longrightarrow \cdots\] of strongly Gorenstein flat modules such that $M\cong \Ker(D_{0}\rightarrow D_{-1})$ and $\Hom_{R}(-, \F(R))$ leaves the sequence exact.
\end{definition}
We use $\SG^{2}\F(R)$ to denote the class of all two-degree strongly Gorenstein flat modules. Clearly, $\SGF(R)\subseteq\SG^{2}\F(R)$.

\begin{proposition} \label{prop2}If $M\in \SG^{2}\F(R)$, then $\Ext_{R}^{i}(M, L)=0$ for each module $L$ with finite flat dimension and each integer $i\geq 1$.
\end{proposition}
\begin{proof} We proceed by induction on $n :=fd_{R}(L)<\infty$. Suppose $M$ is a two-degree strongly Gorenstein flat module. Then there exists a short exact sequence $0\rightarrow K\rightarrow D\rightarrow M\rightarrow 0$ such
that $D\in \SGF(R)$, $K\in \SG^{2}\F(R)$ and $\Hom_{R}(-, F)$ leaves the sequence exact for each flat module $F$. Thus $\Ext_{R}^{1}(M, F)=0$ for each flat module $F$. Applying the functor $\Hom_{R}(-, F)$ to the above sequence, we get the following exact sequence
$$0=\Ext^{1}_{R}(D, F)\rightarrow \Ext^{1}_{R}(K, F)\rightarrow \Ext^{2}_{R}(M, F)\rightarrow \Ext_{R}^{2}(D, F)=0,$$ which yield $\Ext_{R}^{1}(K, F)\cong \Ext_{R}^{2}(M, F)$. By the above proof for $M$, we have $\Ext_{R}^{1}(K, F)=0$, and so $\Ext_{R}^{2}(M, F)=0$. Reiterating this process, we get $\Ext_{R}^{i}(M, F)=0$. Then the case $n=0$ holds. Now suppose $n\geq1$ and $L$ is a module of flat dimension $n$. Let $0\rightarrow L^{'}\rightarrow F\rightarrow L\rightarrow 0$ be an exact sequence such that $F$ is flat. Applying the functor $\Hom_{R}(M, -)$ to it, we get the following exact sequence
$$0=\Ext_{R}^{i}(M, F)\rightarrow \Ext_{R}^{i}(M, L)\rightarrow \Ext_{R}^{i+1}(M, L^{'})\rightarrow \Ext_{R}^{i+1}(M, F)=0.$$
 By inductive assumptions, $\Ext_{R}^{i}(M, L)\cong\Ext_{R}^{i+1}(M, L^{'})=0$
for each integer $i\geq1$, as desired.
\end{proof}

Recall from Definition 2.1 in \cite{BM2007} that a module $M$ is called strongly Gorenstein projective if there exists an exact sequence \[\cdots \stackrel{f}\longrightarrow P\stackrel{f}\longrightarrow P\stackrel{f}\longrightarrow P\stackrel{f}\longrightarrow P\stackrel{f}\longrightarrow \cdots\] of projective modules such that $M\cong \Ker(f)$ and $\Hom_{R}(-, \p(R))$ leaves the sequence exact. It is proved that
each Gorenstein projective module is a direct summand of a strongly Gorenstein projective module (Theorem 2.7 in \cite{BM2007}). Inspired by it, we introduce the notion of strongly two-degree strongly Gorenstein flat modules. This notion plays a crucial role in the proof of the main theorem (see Theorem \ref{th}).

\begin{definition}\label{def2} A module $M$ is called strongly two-degree strongly Gorenstein flat if there exists an exact sequence \[\cdots \stackrel{f}\longrightarrow D\stackrel{f}\longrightarrow D\stackrel{f}\longrightarrow D\stackrel{f}\longrightarrow D\stackrel{f}\longrightarrow \cdots\] of strongly Gorenstein flat modules such that $M\cong \Ker(f)$ and $\Hom_{R}(-, \F(R))$ leaves the sequence exact.
\end{definition}
We use $\SSG^{2}\F(R)$ to denote the class of all strongly two-degree strongly Gorenstein flat modules. Clearly, $\SSG^{2}\F(R)\subseteq\SG^{2}\F(R)$.
\begin{proposition}\label{prop3} For any module $M$, the following statements are equivalent.

$\mathrm{(1)}$ $M\in \SSG^{2}\F(R)$.

$\mathrm{(2)}$ There is a short exact sequence $0\rightarrow M\rightarrow D\rightarrow M \rightarrow 0$ such that $D\in \SGF(R)$ and $\Ext_{R}^{1}(M, F)=0$ for each flat module $F$.

$\mathrm{(3)}$ There is a short exact sequence $0\rightarrow M\rightarrow D\rightarrow M \rightarrow 0$ such that $D\in \SGF(R)$ and $\Ext_{R}^{1}(M, L)=0$ for each module $L$ with finite flat dimension.

$\mathrm{(4)}$ There is a short exact sequence $0\rightarrow M\rightarrow D\rightarrow M \rightarrow 0$ such that $D\in \SGF(R)$ and $\Hom_{R}(-, \F(R))$ leaves the sequence exact.

$\mathrm{(5)}$ There is a short exact sequence $0\rightarrow M\rightarrow D\rightarrow M \rightarrow 0$ such that $D\in \SGF(R)$ and $\Hom_{R}(-, L)$ leaves the sequence exact for each module $L$ with finite flat dimension.

\end{proposition}
\begin{proof}Using standard argument, it follows immediately from the definition of strongly two-degree strongly Gorenstein flat modules.
\end{proof}

\begin{proposition} \label{prop4} Let $M$ be a two-degree strongly Gorenstein flat module. Then $M$ is a direct summand of a strongly two-degree strongly Gorenstein flat module.
\end{proposition}
\begin{proof} Let $M$ be a two-degree strongly Gorenstein flat module. Then there exists an exact sequence \[\cdots \stackrel{\delta_{2}}\longrightarrow D_{1}\stackrel{\delta_{1}}\longrightarrow D_{0}\stackrel{\delta_{0}}\longrightarrow D_{-1}\stackrel{\delta_{-1}}\longrightarrow D_{-2}\stackrel{\delta_{-2}}\longrightarrow \cdots\] of strongly Gorenstein flat modules such that $M\cong \Ker(\delta_{0})$ and $\Hom_{R}(-, \F(R))$ leaves the sequence exact. Consider the exact sequence \[\cdots \stackrel{\bigoplus\delta_{i}}\longrightarrow \bigoplus D_{i}\stackrel{\bigoplus\delta_{i}}\longrightarrow \bigoplus D_{i}\stackrel{\bigoplus\delta_{i}}\longrightarrow \bigoplus D_{i}\stackrel{\bigoplus\delta_{i}}\longrightarrow \bigoplus D_{i} \stackrel{\bigoplus\delta_{i}}\longrightarrow \cdots.\]
Since $\Ker(\bigoplus\delta_{i})\cong \bigoplus\Ker(\delta_{i})$, $M$ is a direct summand of $\Ker(\bigoplus\delta_{i})$. By Lemma \ref{lem1} and $\Hom_{R}(\bigoplus_{i}D_{i}, F)\cong \prod_{i}\Hom_{R}(D_{i}, F)$ for each flat module $F$, we get $\Ker(\bigoplus\delta_{i})$ is a strongly two-degree strongly Gorenstein flat module.
\end{proof}
\begin{theorem} \label{th}$\SGF(R)=\SG^{2}\F(R)$.
\end{theorem}
\begin{proof} Clearly, $\SGF(R)\subseteq\SG^{2}\F(R)$. It suffices to prove that $\SG^{2}\F(R)\subseteq \SGF(R)$. Since $\SGF(R)$ is closed under direct summands, it suffices to prove that any strongly two-degree strongly Gorenstein flat module is strongly Gorenstein flat by Proposition \ref{prop4}. Suppose $M$ is a strongly two-degree strongly Gorenstein flat module. Then there exists a short exact sequence $0\rightarrow M\rightarrow D\rightarrow M \rightarrow 0$ such that $D\in \SGF(R)$ and $\Ext_{R}^{1}(M, F)=0$ for each flat module $F$ by Proposition \ref{prop3}. As $D$ is strongly Gorenstein flat, there is a short exact sequence $0\rightarrow D\rightarrow P\rightarrow D_{1} \rightarrow 0$ such that $P\in \p(R)$ and $D_{1}\in \SGF(R)$. Then we get the following pushout diagram:
 \begin{center}
$\xymatrix{
      & & 0\ar[d]_{}  & 0 \ar[d]_{} &  \\
       0\ar[r]& M\ar@{=}[d]^{} \ar[r]& D\ar[d]\ar[r] &M\ar[d]\ar[r]&0 \\
 0\ar[r]& M \ar[r]& P\ar[d]\ar[r] &N\ar[d]\ar[r]&0 \\
&&D_{1}\ar[d]_{} \ar@{=}[r]^{} & D_{1} \ar[d]_{} \\
    & & 0& 0  &
      }$
\end{center}
For any flat module $F$, applying the functor $\Hom_{R}(-, F)$ to the exact sequence $0\rightarrow M\rightarrow N\rightarrow D_{1} \rightarrow 0$, we get the following exact sequence
$$0=\Ext_{R}^{i}(D_{1}, F)\rightarrow \Ext_{R}^{i}(N, F)\rightarrow \Ext_{R}^{i}(M, F)\rightarrow \Ext_{R}^{i+1}(D_{1}, F)=0$$
for each integer $i\geq 1$. This yields that $\Ext_{R}^{i}(N, F)\cong \Ext_{R}^{i}(M, F)$. By Proposition \ref{prop2}, we have $\Ext_{R}^{i}(M, F)=0$, and so $\Ext_{R}^{i}(N, F)=0$.
 Consider the following pushout diagram:
 \begin{center}
$\xymatrix{
     &&  0\ar[d]_{}  & 0 \ar[d]_{} &  \\
      & 0 \ar[r]^{} & M \ar[d]_{} \ar[r]^{} & D\ar[d]_{} \ar[r]^{} & M \ar@{=}[d]_{} \ar[r]^{} & 0  \\
      &0  \ar[r]^{} & N  \ar[d]\ar[r]^{} &D_{2}\ar[d]_{} \ar[r]^{} & M  \ar[r]^{} & 0  \\
      && D_{1}\ar@{=}[r]^{} \ar[d]& D_{1}\ar[d] & \\
    & & 0& 0  &
      }$

\end{center}
Because both $D$ and $D_{1}$ are strongly Gorenstein flat, $D_{2}$ is also strongly Gorenstein flat by Lemma \ref{lem1}. Then there exists a short exact sequence $0\rightarrow D_{2}\rightarrow P_{0}\rightarrow W\rightarrow 0$ with $P_{0}$ projective and $W$ strongly Gorenstein flat.
Consider the following pushout diagram:
\begin{center}
$\xymatrix{
      & & 0\ar[d]_{}  & 0 \ar[d]_{} &  \\
       0\ar[r]& N\ar@{=}[d]^{} \ar[r]& D_{2}\ar[d]\ar[r] &M\ar[d]\ar[r]&0 \\
 0\ar[r]& N \ar[r]& P_{0}\ar[d]\ar[r] &G\ar[d]\ar[r]&0 \\
&&W\ar[d]_{} \ar@{=}[r]^{} & W \ar[d]_{} \\
    & & 0& 0  &
      }$
\end{center}
Applying the functor $\Hom_{R}(-, F)$ to the exact sequence $0\rightarrow M\rightarrow G\rightarrow W \rightarrow 0$, we get $\Ext_{R}^{i}(G, F)=0$.
On the other hand, applying the functor $\Hom_{R}(-, F)$ to the exact sequence $0\rightarrow N\rightarrow P_{0}\rightarrow G \rightarrow 0$, we get the following exact sequence
$$0\rightarrow \Hom_{R}(G, F)\rightarrow \Hom_{R}(P_{0}, F)\rightarrow \Hom_{R}(N, F)\rightarrow 0.$$
Thus we obtain that an $\Hom_{R}(-, F)$ exact exact sequence $0\rightarrow N\rightarrow P_{0}\rightarrow G \rightarrow 0$ where $P_{0}$ is projective and $G$ is a module with the same property as $N$. Recursively, we get an exact sequence $$0\rightarrow N\rightarrow P_{0}\rightarrow P_{-1}\rightarrow \cdots$$
of projective modules, which remains exact after applying the functor $\Hom_{R}(-, \F(R))$. Thus $N$ is strongly Gorenstein flat. Because both $N$ and $D_{1}$ are strongly Gorenstein flat, $M$ is also strongly Gorenstein flat by Lemma \ref{lem1}.
\end{proof}

\begin{remark} Denote by $\SG^{2}_{G}\F(R)$ the subcategory of all modules
for which there exists an  exact sequence \[\cdots \longrightarrow D_{1}\longrightarrow D_{0}\longrightarrow D_{-1}\longrightarrow D_{-2}\longrightarrow \cdots\] of strongly Gorenstein flat modules such that $M\cong \Ker(D_{0}\rightarrow D_{-1})$ and $\Hom_{R}(-, H)$ leaves the sequence exact for each Gorenstein flat module $H$. It is routine to check that $\SGF(R)\subseteq \SG^{2}_{G}\F(R)\subseteq \SG^{2}\F(R)$. By Theorem \ref{th}, $\SGF(R)=\SG^{2}_{G}\F(R)=\SG^{2}\F(R)$.
\end{remark}

\begin{remark} Denote by GP$(R)$ and G$^{2}$P$(R)$ the subcategories of all Gorenstein projective and two-degree Gorenstein projective modules, respectively.
Denote by $\G^{2}_{G}P(R)$ the subcategory of all modules
for which there exists an  exact sequence \[\cdots \longrightarrow G_{1}\longrightarrow G_{0}\longrightarrow G_{-1}\longrightarrow G_{-2}\longrightarrow \cdots\] of Gorenstein projective modules such that $M\cong \Ker(G_{0}\rightarrow G_{-1})$ and $\Hom_{R}(-, G)$ leaves the sequence exact for each Gorenstein projective module $G$. Denote by $\G^{2}_{P}P(R)$ the subcategory of all modules
for which there exists an  exact sequence \[\cdots \longrightarrow G_{1}\longrightarrow G_{0}\longrightarrow G_{-1}\longrightarrow G_{-2}\longrightarrow \cdots\] of Gorenstein projective modules such that $M\cong \Ker(G_{0}\rightarrow G_{-1})$ and $\Hom_{R}(-, P)$ leaves the sequence exact for each projective module $P$. It is routine to check that $\GP(R)\subseteq \G^{2}P(R) \subseteq \G^{2}_{G}P(R)\subseteq \G^{2}_{P}P(R)$. Similar to the proof of Theorem \ref{th}, we obtain $\G^{2}_{P}P(R)=\GP(R)$. So $\GP(R)=\G^{2}P(R)=\G^{2}_{G}P(R)=\G^{2}_{P}P(R)$.
\end{remark}

According to \cite{Mao2008}, a left $R$-module $M$ is called Gorenstein FP-injective if there is an exact sequence
\[\cdots \longrightarrow E_{1}\longrightarrow E_{0}\longrightarrow E_{-1}\longrightarrow E_{-2}\longrightarrow \cdots\]
of injective left $R$-modules with $M\cong\Ker (E_{0}\longrightarrow E_{-1})$ such that $\Hom_{R}(E, -)$ leaves the sequence exact whenever $E$ an FP-injective $R$-module. We use $\GFI(R)$ to denote the class of all Gorenstein FP-injective modules.

By definitions, every injective module is Gorenstein FP-injective, and every Gorenstein FP-injective module is Gorenstein injective. If $R$ is left noetherian, then every Gorenstein injective module is Gorenstein FP-injective.

Dual arguments to the above give the following assertions concerning the Gorenstein FP-injective modules.

\begin{definition}\label{} A module $M$ is called two-degree Gorenstein FP-injective if there exists an  exact sequence \[\cdots \longrightarrow D_{1}\longrightarrow D_{0}\longrightarrow D_{-1}\longrightarrow D_{-2}\longrightarrow \cdots\] of Gorenstein FP-injective modules such that $M\cong \Ker(D_{0}\rightarrow D_{-1})$ and $\Hom_{R}(H, -)$ leaves the sequence exact for each FP-injective module $H$.
\end{definition}
We use $\GF^{2}\I(R)$ to denote the class of all two-degree Gorenstein FP-injective modules.
\begin{theorem}$\GFI(R)=\GF^{2}\I(R)$.
\end{theorem}

\vspace{0.3cm} \hspace{-0.8cm}{\bf{Acknowledgement}}

 The authors would like to thank the referee for valuable
suggestions and helpful corrections.

\end{document}